\documentclass[a4paper,10pt]{article}

\usepackage{graphicx}
\usepackage{geometry}
\usepackage{amsthm}
\usepackage{amssymb}
\usepackage{amsmath}
\usepackage{hyperref}
\usepackage{xcolor}
\usepackage{enumerate}
\usepackage{listings}
\usepackage{complexity}
\usepackage{blkarray}
\usepackage{lmodern}

\usepackage[font=small,labelfont=bf]{caption}
\usepackage[font=small,labelfont=normalfont,labelformat=simple]{subcaption}
\graphicspath{{figs/}}

\newtheorem{theorem}{Theorem}

\newtheorem{lemma}[theorem]{Lemma}
\newtheorem{corollary}[theorem]{Corollary}

\newtheorem{conjecture}{Conjecture}

\author
{
Raphael Steiner \thanks{Institute of Theoretical Computer Science, ETH Z\"{u}rich, Switzerland,  \texttt{raphaelmario.steiner@inf.ethz.ch}.
This work was supported by an ETH Postdoctoral Fellowship.}
}

\date{\today}

\title{Asymptotic Equivalence of Hadwiger's Conjecture and its Odd Minor-Variant}

\begin{document}
\maketitle

\begin{abstract}
Hadwiger's conjecture states that every $K_t$-minor free graph is $(t-1)$-colorable. A qualitative strengthening of this conjecture raised by Gerards and Seymour, known as the \emph{Odd Hadwiger's conjecture}, states similarly that every graph with no \emph{odd $K_t$-minor} is $(t-1)$-colorable. For both conjectures, their asymptotic relaxations remain open, i.e., whether an upper bound on the chromatic number of the form $Ct$ for some constant $C>0$ exists. 

We show that if every graph without a $K_t$-minor is $f(t)$-colorable, then every graph without an odd $K_t$-minor is $2f(t)$-colorable. Using this, the recent $O(t\log\log t)$-upper bound of Delcourt and Postle~\cite{del} for the chromatic number of $K_t$-minor free graphs directly carries over to the chromatic number of odd $K_t$-minor-free graphs. This (slightly) improves a previous bound of $O(t(\log \log t)^2)$ for this problem by Delcourt and Postle.  
\end{abstract}

\section{Introduction}
Given a number $t \in \mathbb{N}$, a \emph{$K_t$-expansion} is a graph consisting of vertex-disjoint trees $(T_s)_{s=1}^{t}$ and exactly one connecting edge between any pair of trees $T_s, T_{s'}$ for distinct $s,s' \in \{1,\ldots,t\}$. A graph is said to \emph{contain $K_t$ as a minor} or to \emph{contain a $K_t$-minor} if it admits a subgraph which is a $K_t$-expansion. Hadwiger's conjecture, which may well be seen as one of the most central open problems in graph theory, states the following relation between minor containment and the chromatic number of graphs.
\begin{conjecture}[Hadwiger 1943,~\cite{hadwiger}]
If $G$ is a graph which does not contain $K_t$ as a minor, then $\chi(G) \le t-1$.
\end{conjecture}
Hadwiger's conjecture and many variations of it have been studied in past decades, a very good overview of the developments and partial results until about $2$ years ago is given in the survey article~\cite{survey} by Seymour. The best known asymptotic upper bound on the chromatic number of $K_t$-minor free graphs for a long time remained of magnitude $O(t\sqrt{\log t})$, as proved independently by Kostochka~\cite{kostochka} and Thomason~\cite{thomason} in 1984. However, recently there has been progress. First, in 2019, Norine, Postle and Song~\cite{norine} broke the $t \sqrt{\log t}$ barrier by proving an upper bound of the form $O(t (\log t)^\beta)$ for any $\beta>\frac{1}{4}$. Subsequently, there have been several significant improvements of this bound and related results~\cite{norine2},\cite{postle},\cite{postle2}. The following state of the art-bound was proved recently by Delcourt and Postle in~\cite{del}. 
\begin{theorem}\label{thm:del}
The maximum chromatic number of $K_t$-minor free graphs is bounded from above by a function in $O(t\log \log t)$. 
\end{theorem}
A strengthening of Hadwiger's conjecture to so-called \emph{odd minors} was conjectured by Gerards and Seymour in~\cite{gerards}. A $K_t$-expansion $H$ certified by a corresponding collection of vertex-disjoint trees $(T_s)_{s=1}^{t}$ is said to be \emph{odd}, if there exists an assignment of two colors $\{1, 2\}$ to the vertices of $H$ in such a way that every edge contained in one of the trees $T_s$ with $s \in \{1,\ldots,t\}$ is bichromatic (i.e., has different colors at its endpoints), while every edge joining two distinct trees is monochromatic (i.e., has the same color at its endpoints).

Finally, we say that a graph contains \emph{$K_t$ as an odd minor} or that it contains an \emph{odd $K_t$-minor} if it contains a subgraph which is an odd $K_t$-expansion. 
\begin{conjecture}[Gerards and Seymour~\cite{gerards}]\label{oddconjecture}
If $G$ is a graph which does not contain $K_t$ as an odd minor, then $\chi(G) \le t-1$.
\end{conjecture}
Just as for Hadwiger's conjecture, the asymptotic growth of the best-possible upper bound on the chromatic number of graphs without an odd $K_t$-minor has been studied. First, Geelen, Gerards, Reed, Seymour and Vetta proved in~\cite{geelen} that every graph with no odd $K_t$-minor is $O(t \sqrt{\log t})$-colorable. A shorter proof for the same result was given by Kawarabayashi in~\cite{kawarabayashi}.  Subsequently, an asymptotical improvement of this upper bound to $O(t(\log \log t)^\beta)$ for any $\beta>\frac{1}{4}$ was achieved by Norine and Song in~\cite{norine3}. This was improved further to $O(t (\log \log t)^6)$ by Postle in~\cite{postle3}. Very recently the exponent of the $\log \log t$ factor was further improved by Delcourt and Postle in~\cite{del}, resulting in an $O(t(\log \log t)^2)$-bound. 
Many further results on odd $K_t$-minor free graphs are known, we refer to~\cite{kang,kawarabayashi2,kawarabayashi3,kawarabayashireed,kawarabayashireedwollan,kawarabayashisong} for some additional references.

The purpose of this note is to show that asymptotically, the maximum chromatic number of $K_t$-minor free graphs and the maximum chromatic number of odd $K_t$-minor free graphs differ at most by a multiplicative factor of $2$.
 
\begin{theorem}\label{thm:main}
Let $t \in \mathbb{N}$ and let $f(t)$ be an integer such that every graph not containing $K_t$ as a minor is $f(t)$-colorable. Then every graph not containing $K_t$ as an odd minor is $2f(t)$-colorable. 
\end{theorem}
Theorem~\ref{thm:main} has a very simple proof, given in Section~\ref{sec:proof} below. It is useful in the sense that any progress made towards a better asymptotic upper bound on the chromatic number of $K_t$-minor free graphs carries over, without further work and only at the prize of a constant multiplicative factor, to odd $K_t$-minor free graphs. In particular, Theorem~\ref{thm:del} together with Theorem~\ref{thm:main} directly yields the following (slight) asymptotical improvement of the $O(t(\log \log t)^2)$-upper bound on the chromatic number of odd $K_t$-minor free graphs by Delcourt and Postle. 
\begin{corollary}
The maximum chromatic number of odd $K_t$-minor free graphs is bounded from above by a function in $O(t\log \log t)$. 
\end{corollary}
\section{Proof of Theorem~\ref{thm:main}}\label{sec:proof}
The proof is based on the following lemma.

\begin{lemma}\label{lemma}
Let $G$ be a graph. Then there exists $n \in \mathbb{N}$ and a partition of $V(G)$ into $n$ non-empty sets $X_1,\ldots,X_n$ such that the following hold: 
\begin{itemize}
\item for every $1 \le i \le n$, the graph $G[X_i]$ is bipartite and connected,
\item for every $1 \le i<j \le n$, either there are no edges in $G$ between $X_i$ and $X_j$, or there exist $u_1, u_2 \in X_i$ and $v \in X_j$ such that $u_1v, u_2v \in E(G)$ and $u_1$ and $u_2$ lie on different sides of the bipartition of $G[X_i]$.
\end{itemize}
\end{lemma}
\begin{proof}
Let us define the partition $X_1,X_2,\ldots$ of $V(G)$ inductively as follows:

Suppose that for some integer $i \ge 1$, all the sets $X_k$ with $1 \le k<i$ have been defined already, and do not yet form a partition, i.e., $\bigcup_{1 \le k<i}{X_k} \neq V(G)$. We now choose $X_i$ as an inclusion-wise maximal set among all subsets $X \subseteq V(G) \setminus \bigcup_{1 \le k<i}{X_k}$ which satisfy that $G[X]$ is a bipartite and connected graph. Note that $X_i \neq \emptyset$, since for every vertex $x \in V(G) \setminus \bigcup_{1 \le k<i}{X_k}$, the graph $G[\{x\}]$ is bipartite and connected. 

Since we are adding a non-empty set to our collection of pairwise disjoint subsets of $V(G)$ at each step, the above procedure eventually yields a partition $X_1,\ldots,X_n$ of $V(G)$ for some $n \in \mathbb{N}$. By definition, we have that $G[X_i]$ is bipartite and connected for $i=1,\ldots,n$, and hence what remains to show is the second property of the partition stated in the lemma. 

So let $i,j$ be given such that $1 \le i <j \le n$, and suppose that there exists at least one edge $e \in E(G)$ between $X_i$ and $X_j$. Denote $e=uv$ with $u \in X_i$ and $v \in X_j$. Let $\{A; B\}$ be the unique bipartition of $G[X_i]$. We claim that $v$ must have a neighbor $u_1 \in A$ and a neighbor $u_2 \in B$, which then yields the statement claimed in the lemma. Indeed, suppose not, and suppose w.l.o.g. that $v$ is not adjacent to any vertex in $A$ (the case that $v$ has no neighbor in $B$ is of course symmetric). Then also the graph $G[X_i \cup \{v\}]$ is bipartite and connected: It is connected since $G[X_i]$ is connected and because of the edge $uv$, and it is bipartite since $\{A \cup \{v\}; B\}$ forms its unique bipartition. However, putting $X:=X_i \cup \{v\} \subseteq V(G) \setminus \bigcup_{1 \le k<i}{X_k}$, this contradicts the definition of $X_i$ as an inclusion-wise maximal subset of $V(G) \setminus \bigcup_{1 \le k<i}{X_k}$ inducing a bipartite and connected subgraph.
\end{proof}

We can now easily deduce Theorem~\ref{thm:main}. 

\begin{proof}[Proof of Theorem~\ref{thm:main}]
Let $t \in \mathbb{N}$ and suppose that $f(t)$ is an integer such that every $K_t$-minor free graph is $f(t)$-colorable. Let $G$ be any given graph without an odd $K_t$-minor, and let us prove that $\chi(G) \le 2f(t)$. 

We apply Lemma~\ref{lemma} to $G$ and obtain a partition $X_1,\ldots,X_n$ of $V(G)$ with properties as stated in the lemma. Let $H$ be defined as the graph with vertex-set $\{1,\ldots,n\}$ and which has an edge between distinct vertices $i$ and $j$ if and only if there exists at least one edge in $G$ between $X_i$ and $X_j$. It follows from the statement of the lemma that for every edge $ij \in E(H)$ with $i<j$, there exist $u_1, u_2 \in X_i$ and $v \in X_j$ such that $u_1v, u_2v \in E(G)$ and $u_1, u_2$ are on different sides of the bipartition of $G[X_i]$. 

We claim that $\chi(G) \le 2\chi(H)$. To see this, let $c_H:\{1,\ldots,n\} \rightarrow \{1,\ldots,\chi(H)\}$ be a proper coloring of $H$, and for every $i \in \{1,\ldots,n\}$ let $c_i:X_i \rightarrow \{1,2\}$ be a proper coloring of the bipartite graph $G[X_i]$. It now follows directly from the definition of $H$ that the coloring $c_G$ of $G$ with color set $\{1,\ldots,\chi(H)\} \times \{1,2\}$, defined by $c_G(x):=(c_H(i),c_i(x))$ for every $x \in X_i$ and $i \in \{1,\ldots,n\}$, is a proper coloring of $G$. Therefore, $\chi(G) \le 2\chi(H)$.

Next, we will show that $\chi(H) \le f(t)$ by proving that $H$ does not contain $K_t$ as a minor. Suppose towards a contradiction that $H$ contains a subgraph which is a $K_t$-expansion, i.e., there exist vertex-disjoint trees $(T_s)_{s=1}^{t}$ contained in $H$ and for every pair $\{s,s'\} \subseteq \{1,\ldots,t\}$ an edge $e(s,s') \in E(H)$ with endpoints in $T_s$ and $T_{s'}$. 

For every fixed $s \in \{1,\ldots,t\}$, let us consider the subgraph $G_s:=G\left[\bigcup_{i \in V(T_s)}{X_i}\right]$ of $G$. This is a connected graph because $G[X_i]$ is connected for every $i \in V(T_s)$, since $T_s$ is connected, and since by definition of $H$ for every edge $ij \in E(T_s)$ there exists at least one connecting edge between $X_i$ and $X_j$ in $G$. In particular, $G_s$ contains a spanning tree $T^G_s$ which has the property that $T^G_s[X_i]$ forms a spanning tree of $G[X_i]$, for every $i \in V(T_s)$.

The trees $(T^G_s)_{s=1}^{t}$ in $G$ defined as above are pairwise vertex-disjoint. Let us denote by $c:\bigcup_{s=1}^{t}{V(T^G_s)} \rightarrow \{1,2\}$ a $2$-color-assignment obtained by piecing together proper $2$-colorings of the individual trees $(T^G_s)_{s=1}^{t}$. 

\paragraph{Claim.} For every pair $\{s,s'\} \subseteq \{1,\ldots,t\}$, there exists an edge $f(s,s') \in E(G)$ with endpoints in $T^G_{s}$ and $T^G_{s'}$, such that $f(s,s')$ is monochromatic with respect to the coloring $c$.
\begin{proof}[Subproof]
By assumption, there exists $e(s,s') \in E(H)$ which connects a vertex in $i \in V(T_s)$ to a vertex $j \in V(T_{s'})$. Possibly after relabelling assume w.l.o.g. $1 \le i <j \le n$. Then the second property of the partition $X_1,\ldots,X_n$ guaranteed by Lemma~\ref{lemma} yields the existence of vertices $u_1, u_2 \in X_i \subseteq V(T^G_{s})$ and $v \in X_j\subseteq V(T^G_{s'})$ such that $f_1:=u_1v, f_2:=u_2v \in E(G)$, and such that $u_1$ and $u_2$ lie on different sides of the unique bipartition of $G[X_i]$. Since $u_1, u_2 \in X_i \subseteq V(T^G_{s})$ and since by our choice of $T^G_s$ the graph $T^G_s[X_i]$ forms a spanning tree of $G[X_i]$, it follows that $u_1$ and $u_2$ also must be on different sides in the unique bipartition of $T^G_s[X_i]$. In particular, $c(u_1) \neq c(u_2)$, which implies that $c(u_r)=c(v)$ for some $r \in \{1,2\}$. Now the edge $f_r \in E(G)$ connects the vertex $u_r$ in $T^G_s$ with the vertex $v$ in $T^G_{s'}$, and is monochromatic with respect to $c$. This proves the subclaim with $f(s,s'):=f_r$.
\end{proof}
It follows directly from the previous claim that the union of the vertex-disjoint trees $(T^G_s)_{s=1}^{t}$, joined by the edges $f(s,s')$ for every pair $\{s,s'\} \subseteq \{1,\ldots,t\}$, forms an odd $K_t$-expansion contained in $G$. This contradicts our initial assumption that $G$ does not contain $K_t$ as an odd minor. Hence, our initial assumption was wrong, and we have established that $H$ is $K_t$-minor free. It now follows that $\chi(G) \le 2\chi(H)\le 2f(t)$, as required. This concludes the proof of the theorem. 
\end{proof}

\end{document}